\documentclass[amsfonts, reqno]{amsart}

\usepackage{amssymb,amsfonts}
\usepackage[all,arc]{xy}
\usepackage{enumerate}
\usepackage{mathrsfs}
\usepackage{algorithm}
\usepackage{algpseudocode}

\usepackage{amsmath}
\usepackage{amssymb}
\usepackage{amsthm}
\usepackage{amsrefs}
\usepackage{mathrsfs}
\usepackage{mathtools}
\usepackage{hyperref}
\usepackage{pgfplots}
\usepackage{tikz}
\usepackage{setspace}
\usepackage{graphicx}
\usepackage{subcaption}
\usepackage{fancyhdr}
\usetikzlibrary{positioning}
\usetikzlibrary{calc}
\pgfplotsset{compat=1.12}

\usepackage[shortlabels]{enumitem}
\usepackage{verbatim} 

\newtheorem{thm}{Theorem}[section]
\newtheorem{cor}[thm]{Corollary}

\newtheorem{lem}[thm]{Lemma}

\theoremstyle{definition}
\newtheorem{defn}[thm]{Definition}

\theoremstyle{remark}
\newtheorem{rem}[thm]{Remark}

\newcommand{\bbR}{\mathbb{R}}

\newcommand{\bbN}{\mathbb{N}}
\newcommand{\dist}{\text{dist}}

\makeatletter
\let\c@equation\c@thm
\makeatother
\numberwithin{equation}{section}
\title{Convergence of Reflected Langevin Diffusion for Constrained Sampling}

\author{Tarika Mane, Amine Boukardagha}

\date{January, 2026}

\begin{document}
\setstretch{1.25}

\begin{abstract} We examine the Langevin diffusion confined to a closed, convex domain $D\subset\mathbb{R}^d$, represented as a reflected stochastic differential equation. We introduce a sequence of penalized stochastic differential equations and prove that their invariant measures converge, in Wasserstein--2 distance and with explicit polynomial rate, to the invariant 
measure of the reflected Langevin diffusion. We also analyze a 
time-discretization of the penalized process obtained via the Euler–Maruyama 
scheme and demonstrate the convergence to the original constrained measure. These results provide a rigorous approximation framework for reflected Langevin dynamics in both continuous and discrete time.
\end{abstract}

\maketitle
\thispagestyle{empty}
\section{Introduction}
In this work, we consider the overdamped Langevin stochastic differential equation given by
\begin{equation} \label{diffusion}
dX_t = -\nabla f(X_t)\,dt + \sigma\,dW_t,
\end{equation}
where $f: \bbR^d \to \bbR$ is a potential energy function, $\sigma > 0$ is a diffusion coefficient, and $ W_t $ is a standard Brownian motion in $ \mathbb{R}^d $. Under suitable assumptions on $f$, the process $(X_t)_{t \geq 0} $ admits a unique invariant Gibbs measure where the density is proportional to $ e^{-\frac{2}{\sigma^2}f(x)}$.

The Langevin Monte Carlo algorithm is a classical method for sampling from a target probability measure $\pi$.  
If we set 
\[
f(x) \coloneqq -\tfrac{\sigma^2}{2}\log \pi(x),
\]
then the Langevin equation
\[
dX_t = \nabla \log \pi(X_t)\,dt + \sigma\,dW_t
\]
admits $\pi$ as its unique invariant measure.  
Discretizing the SDE \eqref{diffusion} with the Euler--Maruyama scheme and step size $h>0$ yields
\begin{equation*}
X_{k+1} = X_k - h \nabla f(X_k) + \sigma\sqrt{h}\,\xi_k, 
\qquad \xi_k \sim \mathcal{N}(0,I_d),
\end{equation*}
which is known as the Unadjusted Langevin Algorithm (ULA).

If the target measure $\pi$ is supported on a closed, convex domain $D \subset \bbR^d$, then the Langevin diffusion process may evolve outside of $D$, leading to recent work on constrained Langevin Monte Carlo methods. Existing approaches enforce the constraint through a correction mechanism, such as projection onto $D$, Metropolis--Hastings accept--reject steps, or the use of mirror maps \cites{Bubeck2015, Lamperski2021-2, Hsieh2020}.  
In this paper, we instead consider the constrained Langevin diffusion defined by a reflected stochastic differential equation, a framework studied in \cites{slominski2013, Slominski2001} for reflected It\^{o} diffusion processes and more classically in \cite{Tanaka1979}. A key advantage of the reflected Langevin equation is that the constraint is enforced intrinsically through a zero--flux boundary condition on $\partial D$.  Consequently, the reflected process preserves the Gibbs measure on $D$ as its unique invariant measure and avoids the discretization bias introduced by projection-based schemes and eliminates the need for Metropolis corrections.

To analyze the reflected Langevin diffusion, we approximate it by a family of 
penalized SDEs indexed by a penalty parameter $n$.  
Each penalized process evolves in the full space $\mathbb{R}^d$ and admits an invariant measure that assigns increasing weight to the domain $D$ 
as $n\to\infty$. We establish that the invariant measures of the penalized SDEs converge, in Wasserstein distance, to the invariant measure of the reflected Langevin diffusion. We then apply the Euler--Maruyama scheme and get the \textbf{Penalized Constrained Unadjusted Langevin Algorithm (PCULA)}:
\begin{equation*} \label{pcula}
X_{k+1}^{(n,h)} = X_k^{(n,h)} - h \nabla f_n\left(X_k^{(n,h)}\right) + \sigma\sqrt{h}\,\xi_k, \quad \xi_k \sim \mathcal{N}(0, I_d)
\end{equation*} for a fixed step size $h > 0$ and fixed $n \in \bbN.$
This algorithm thus yields an implementable sampling scheme that approximates 
the invariant measure of the constrained Langevin diffusion.

\section{Reflected Langevin Diffusion via Penalization}
When considering stochastic dynamics constrained to a closed convex domain $D \subset \bbR^d$, the current Langevin diffusion process \eqref{diffusion} will evolve freely and can exit $D$. To deal with this issue, we reflect at the boundary to enforce $X_t \in D$ for all $t \geq 0$, by introducing an additional term of bounded variation that acts only when the trajectory reaches $\partial D$. Inside the domain, the evolution follows the usual Langevin diffusion and this construction yields the reflected Langevin stochastic differential equation
\begin{equation} \label{X}
X_t =X_0 + \int_{0}^{t} \sigma dW_s -\int_{0}^{t}\nabla g(X_s) ds +K_t.
\end{equation}
Here, we start from $X_0 \in D$, $X_t$ is constrained to $D$, $K_t$ is a bounded variation process with variation $||K||$ increasing only when $X_t \in \partial D$, $W_t$ is a standard $d$-dimensional Brownian motion, $\sigma > 0$ is a diffusion coefficient, and $g: \bbR^d \to \bbR$ is the potential energy function. 

Direct analysis of the reflected SDE \eqref{X} is challenging 
due to the presence of the boundary reflection term $K_t$.  
A standard approach is to replace the hard constraint $X_t \in D$ with a 
soft constraint obtained through penalization. In this formulation, the reflection is approximated by a restoring force that 
acts whenever the process moves outside the domain, with strength controlled by 
a penalty parameter. As the penalty strength grows, the trajectories remain increasingly close to $D$ and the goal is to converge on the invariant measure of the reflected Langevin SDE. Formally, for each $n \in \bbN,$ we define the penalty term as \begin{equation} \label{K}
K_t^n = -n \int_{0}^{t} \left( X_s^n - \Pi(X_s^n) \right) \,ds
\end{equation} where $\Pi(x)$ is the metric projection of $x$ to $D$. For a closed convex set $D$, the metric projection $\Pi:\mathbb{R}^d\to D$ is
single-valued, non-expansive, and characterized by the property that
$x-\Pi(x)$ is the minimal vector pushing $x$ back toward the domain.
Thus the penalization term $n(x-\Pi(x))$ provides the natural convex-analytic
restoring force that approximates the reflection mechanism in the limit
$n\to\infty$. This gives us the penalized process $X_t^n$ satisfying \begin{equation} \label{Xn}
X_t^n =X_0 + \int_{0}^{t} \sigma dW_s -\int_{0}^{t}\nabla g(X_s^n) ds +K_t^n.
\end{equation}
One can allow the Brownian motion $W_t$ to vary for each $n \in \bbN$, creating a new penalized equation with $W_t^n$; for our purposes, we keep as $W_t^n \coloneqq W_t$ as the standard $d$-dimensional Brownian motion on $\bbR^d.$ 

The penalized process \eqref{Xn} can then be reconsidered as the usual Langevin diffusion process \eqref{diffusion} with a modified potential energy function. This identification enables us to study $X_t^n$ via the well-developed theory 
of Langevin diffusion.

\begin{thm}\label{langevin form}
Let $D \subset \bbR^d$ be closed and convex, and $g \in C^1(\bbR^d)$ be our energy function. For $n \geq 1$, define $$f_n(x)= g(x) + \frac{n}{2} \dist^2(x, D).$$ Suppose $X^n$ satisfies the penalized SDE \eqref{Xn} for a Brownian motion $W^n$, with $K^n_t$ defined by \eqref{K}. Then, $X^n$ is a solution of the Langevin SDE \begin{equation} \label{standard_penalized}
dX_t^n = -\nabla f_n(X_t^n) dt + \sigma dW_t^n.
\end{equation}
Conversely, any solution of \eqref{standard_penalized} satisfies \eqref{Xn} (with $K^n$ given by \eqref{K}).
We refer to $f_n$ as the penalized energy function.
\end{thm}
\begin{proof}
    For a closed convex domain $D$, the metric projection $\Pi(x)$ is single-valued and $1$-Lipschitz. The function $\psi(x) = \frac{1}{2}\dist^2(x,D)$ is $C^1(\bbR^d)$ with \[
    \nabla\psi(x) = x - \Pi(x)
    \] for all $x \in \bbR^d$ \cite{Balestro2019}. Consequently, $\nabla\psi$ is $1$-Lipschitz. From our integral equation \eqref{Xn} and definition of $K^n$, we get 
    \begin{eqnarray*}
        X^n_t & = & X_0 + \int_{0}^{t} \sigma dW_s^n -\int_{0}^{t}\nabla g(X_s^n) ds -n\int_0^t (X_s^n - \Pi(X_s^n)) ds \\
        & = & X_0 + \int_{0}^{t} \sigma dW_s^n - \int_0^t\left[ \nabla g(X_s^n) + n\nabla\psi(X_s^n) \right] ds \\
        & = & X_0 + \int_0^t \sigma dW_s^n - \int_0^t \nabla f_n(X_s^n) ds.
    \end{eqnarray*}
    This is precisely the integral form of \eqref{standard_penalized}. In differential notation, we have \[
    dX_t^n = -\nabla f_n(X_t^n) dt + \sigma dW_t^n.
    \]
    Conversely, if $X^n$ is a solution to \eqref{standard_penalized} then integrating from $0$ to $t$ gives us \[
    X_t^n = X_0 + \int_{0}^{t} \sigma dW_s^n -\int_{0}^{t}\nabla g(X_s^n) ds -n\int_0^t \nabla\psi(X_s^n) ds.
    \]
\end{proof}

Note that $\Pi(x) = x$ for $x \in D$ so $\nabla\psi(x) = 0$ inside $D$ which means the penalty vanishes, and outside of $D$, the restoring term is $-n(x - \Pi(x))$ which pulls the process back towards the domain.
\begin{rem}
    For closed convex $D$, we have $\psi \coloneqq \frac{1}{2}\dist^2(x,D) \in C^1(\bbR^d)$ and $\nabla\psi(x) = x - \Pi(x)$ is 1-Lipschitz; hence if $\nabla g$ is $L$-Lipschitz then the penalized energy function $\nabla f_n$ is $(L + n)$-Lipschitz.
\end{rem}
We are also able to obtain the convexity of $f_n$ based on the convexity of $g.$ The Lipschitz and convexity conditions on the energy functions are of particular significance as they allow us to obtain unique invariant probability measures for the reflected and penalized processes, which is explored in Section 3.
\begin{lem}\label{f convex thm}
    Let $D \subset \bbR^d$ be closed and convex, and define \[
    \psi(x) \coloneqq \frac{1}{2}\dist^2(x,D), \quad f_n(x) \coloneqq g(x) + n\psi(x),
    \] with $g \in C^1(\bbR^d).$ If $g$ is $m$-strongly convex then $f_n$ is also $m$-strongly convex on $\bbR^d$.
\end{lem}
\begin{proof}
    Since $D$ is convex, $\psi$ is convex on $\bbR^d$ so for all $x,y \in \bbR^d$, we have \begin{equation}\label{psi convex}
        \psi(y) \geq \psi(x) + \langle \nabla\psi(x), y-x \rangle.
    \end{equation}
    Since $g$ is $m$-strongly convex, we have \begin{equation}\label{g convex}
        g(y) \geq g(x) + \langle \nabla g(x), y-x \rangle + \frac{m}{2}||y-x||^2
    \end{equation} for all $x,y \in \bbR^2$.
    So, since $\nabla f_n = \nabla g + n \nabla\psi$, adding $n$ times \eqref{psi convex} to \eqref{g convex} yields \[
    f_n(y) \geq f_n(x) + \langle \nabla f_n(x), y-x\rangle + \frac{m}{2}||y-x||^2.
    \] Thus, $f_n$ is $m$-strongly convex, for all $n \in \bbN$.
\end{proof}

\section{Convergence to a Unique Invariant Measure in Wasserstein-2 Distance}
Having established the convexity of the penalized energy functions $f_n$, we now turn to the analysis of the convergence properties of the laws of the penalized Langevin SDE \eqref{standard_penalized} to a unique invariant measure. To study the convergence, we introduce the Wasserstein distance on the space of probability measures which quantifies the minimal transport cost to move from one distribution to another. 

Let $\mathcal{B}(\bbR^d)$ denote the Borel $\sigma$-algebra on $\bbR^d$ and let $\mathcal{P}_2(\bbR^d)$ denote the set of all Borel probability measures with finite second moment. 
\begin{defn}[Coupling]
Let $\mu$ and $\nu$ be probability measures on $(\bbR^d, \mathcal{B}(\bbR^d)).$ A \textit{coupling} (or \textit{transport plan}) between $\mu$ and $\nu$ is a probability measure $\zeta$ on $(\bbR^d \times \bbR^d, \mathcal{B}(\bbR^d) \otimes \mathcal{B}(\bbR^d))$ whose marginals are $\mu$ and $\nu$. That is, \[
\zeta(A \times \bbR^d) = \mu(A), \quad \zeta(\bbR^d \times A) = \nu(A) \quad \text{for all measurable }A \subset \bbR^d.
\]  We denote by $\Omega(\mu, \nu)$ the set of all such couplings between $\mu$ and $\nu.$  
\end{defn}

Equivalently, a pair of $\bbR^d$ random variables $(X,Y)$ is said to be a coupling of $\mu$ and $\nu$ if there exists $\zeta \in \Omega(\mu, \nu)$ such that the joint distribution of $(X,Y)$ is $\zeta$, that is, $X \sim \mu$ and $Y \sim \nu.$

\begin{defn}[Wasserstein Distance of order 2] For two probability measures $\mu$ and $\nu$ on $(\bbR^d, \mathcal{B}(\bbR^d)),$ the \textit{Wasserstein distance of order 2} between $\mu$ and $\nu$ is defined as \begin{equation}\label{eq:W2}
W_2(\mu,\nu)
:= \left( \inf_{\zeta \in \Omega(\mu,\nu)}
\int_{\mathbb{R}^d \times \mathbb{R}^d} ||x-y||^2 , d\zeta(x,y) \right)^{1/2}.
\end{equation}
    
\end{defn}

This is also called the \textit{Wasserstein-2 distance}, and $\mathcal{P}_2(\bbR^d)$ forms a complete, separable metric space under the $W_2$ metric. Equivalently, for any coupling $(X,Y)$ of $\mu$ and $\nu$,
\begin{equation}\label{eq:W2_expectation}
W_2(\mu,\nu)
= \inf_{(X,Y)} \big( \mathbb{E}\big[||X-Y||^2\big] \big)^{1/2},
\end{equation} where the infimum is taken over all joint distributions of pairs $(X,Y)$ with respective laws $\mu$ and $\nu$ \cite{Villani2009}.

\begin{thm} \label{penalized contraction}
    Let $\mu_0, \nu_0 \in \mathcal{P}_2(\bbR^d)$ be two probability measures, and let $D \subset \bbR^d$ be closed and convex. Suppose $g \in C^1(\bbR^d)$ is $m$-strongly convex. For some $n \in \bbN$, let $X_t^n, Y_t^n \in D$ denote two respective solutions to the penalized Langevin SDE \eqref{Xn} driven by the same Brownian motion $W_t$ and energy function $g$ with initial distributions $X_0 \sim \mu_0$ and $Y_0 \sim \nu_0.$ Then, for all $t \geq 0$, we have the following contraction result: \begin{equation}
        W_2(\mu_t^n, \nu_t^n) \leq e^{-mt}W_2(\mu_0, \nu_0)
    \end{equation} where $\mu_t^n$ and $\nu_t^n$ denote the laws of $X_t^n$ and $Y_t^n$ respectively.
\end{thm}

\begin{proof}
    Here, we can follow a similar argument to \cite{Bolley2012}. Consider two copies of the penalized diffusion $X_t^n$ and $Y_t^n$ driven by the same Brownian motion $W_t$ starting from $X_0$ with law $\mu_0$ and $Y_0$ with law $\nu_0$ respectively. Then, $X_t^n$ and $Y_t^n$ then satisfy \begin{equation*}
        dX_t^n = -\nabla f_n(X_t^n)dt + \sigma dW_t, \qquad dY_t^n = -\nabla f_n(Y_t^n)dt + \sigma dW_t
    \end{equation*} by Theorem \ref{langevin form}, where $f_n = g(x) + \frac{n}{2}\dist^2(x,D)$. Applying Ito's formula to the squared distance yields \begin{equation*}
        \frac{d}{dt}||X_t^n - Y_t^n||^2 = -2\langle \nabla f_n(X_t^n) - \nabla f_n(Y_t^n), X_t^n - Y_t^n \rangle.
    \end{equation*}
    Since $f_n$ is strongly $m$-convex by Lemma \ref{f convex thm}, then for all $x,y \in \bbR^d$, \begin{equation*}
        m ||x-y||^2 \leq \langle \nabla f_n(x) - \nabla f_n(y), x - y \rangle,
    \end{equation*}
    so we have that \begin{equation*}
        \frac{d}{dt}||X_t^n - Y_t^n||^2 \leq -2m||X_t^n - Y_t^n||^2.
    \end{equation*}
    Next, by Gr\"{o}nwall's Lemma,  \begin{equation*}
        ||X_t^n - Y_t^n||^2 \leq e^{-2mt}||X_0 - Y_0||^2
    \end{equation*}
    so $\mathbb{E}[||X_t^n - Y_t^n||^2] \leq e^{-2mt}\mathbb{E}[||X_0 - Y_0||^2]$, and this holds for any arbitrary coupling $(X_0, Y_0)$ of $\mu_0$ and $\nu_0$. Applying \ref{eq:W2_expectation} gives us
    \begin{equation*}
        \left(W_2(\mu_t^n, \nu_t^n)\right)^2 \leq \mathbb{E}[||X_t^n - Y_t^n||^2],
    \end{equation*}
    and taking the infimum over all couplings of $X_0, Y_0$ with laws $\mu_0, \nu_0$ implies that \begin{equation*}
        \left(W_2(\mu_t^n, \nu_t^n)\right)^2 \leq e^{-2mt}\left(W_2(\mu_0, \nu_0)\right)^2.
    \end{equation*} Taking the square root completes the proof.
\end{proof}

Our primary interest is to use the penalized Langevin SDEs to converge on the invariant measure of the reflected Langevin SDE \eqref{X}. For the Langevin equation driven by a potential energy function $f:\mathbb{R}^d\to\mathbb{R}$,
\[
dX_t = -\nabla f(X_t)\,dt + \sigma\,dW_t,
\]
the infinitesimal generator $L$ acts on test functions 
$\psi\in C^2(\mathbb{R}^d)$ as
\[
L\psi(x) = -\nabla f(x)\cdot\nabla\psi(x)
            + \frac{\sigma^2}{2}\,\Delta\psi(x).
\]
A probability measure $\pi$ is invariant for the process if and only if it
satisfies the stationary Fokker--Planck equation $L^*\pi = 0$, where $L^*$ denotes the $L^2$–adjoint of $L$. Moreover, for the Langevin generator, with respect to the inner product of $L^2(\pi)$, $L$ is self-adjoint \cite{Kostic2024}.

\begin{cor}\label{inv Xn}
Assume $D \subset \bbR^d$ is closed and convex, and $g$ is $m$-strongly convex. For $n \geq 1$, let $X_t^n$ be a solution to the penalized Langevin SDE \eqref{Xn}. Then, there exists a unique invariant probability measure $\pi^n$ with density
\begin{equation*}
    \frac{d\pi^n}{dx}(x) = \frac{1}{Z^n}\exp\left( -\frac{2}{\sigma^2} f_n(x) \right), \quad Z^n = \int_{\bbR^d} \exp\left( -\frac{2}{\sigma^2} f_n(x) \right) dx < \infty.
\end{equation*}
Furthermore, for every initial law $\mu_0 \in \mathcal{P}_2(\bbR^d)$, \begin{equation*}
    W_2(\mu_t^n, \pi^n) \leq e^{-mt}W_2(\mu_0, \pi^n), \quad t \geq 0
\end{equation*}
where $\mu_t^n$ is the law of $X_t^n.$
\end{cor}
\begin{proof}
    First note that by Theorem \ref{langevin form}, $X_t^n$ satisfies $dX_t^n = -\nabla f_n(X_t^n)dt + \sigma dW_t^n$ where $f_n(x) = g(x) + \frac{n}{2}\dist^2(x,D)$. Then, $f_n$ is $m$-strongly convex by Lemma \ref{f convex thm} which implies that $f_n(x) \geq \frac{m}{2}||x||^2 - c$ for some $c > 0$ so \begin{equation*}
        Z^n = \int_{\bbR^d}\exp\left( -\frac{2}{\sigma^2} f_n(x) \right) dx < \infty.
    \end{equation*}
    The Gibbs measure $\pi^n$ is therefore well defined, and it is invariant for the generator $L = -\nabla f_n \cdot \nabla + \frac{\sigma^2}{2} \Delta$ via integration-by-parts.

    Additionally, $X_t^n$ defines a Markov semi-group that is contractive in the $W_2$ distance by Theorem \ref{penalized contraction}. Since $\pi^n$ is invariant for this semi-group, it is unique \cites{Durmus2019,Ambrosio2009}. Taking $\nu_0 = \pi^n$ in Theorem \ref{penalized contraction} gives $W_2(\mu_t^n, \pi^n) \leq e^{-mt}W_2(\mu_0,\pi^n)$.
\end{proof}

From here, we now study the contraction results for the reflected Langevin stochastic differential equation \eqref{X}. We also introduce necessary Lipschitz and geometric drift conditions on the potential energy function to construct the bounds, as done in \cites{Durmus2019, Eberle2019, Lamperski2021, slominski2013, Slominski2001}.

\begin{thm}\label{error contraction}
    Let $D \subset \bbR^d$ be closed and convex. Assume \begin{enumerate}
        \item[(A1)] $g \in C^1(\bbR^d)$ is $m$-strongly convex with $\nabla g$ globally $L$-Lipschitz: \[
        ||\nabla g(x) - \nabla g(y)|| \leq L ||x-y|| \quad \text{ for all } x,y \in \bbR^d.
        \] 
        \item[(A2)] there exists $R > 0$ such that \[
        \sigma^2 + ||\nabla g(x)||^2 \leq R(1 + ||x||^2) \quad \text{ for all } x \in \bbR^d.
        \]
    \end{enumerate} Let $X_t, Y_t \in D$ be solutions of the RSDE \eqref{X}, driven by the same Brownian motion $W_t$, \begin{eqnarray*}
        X_t &= &X_0 + \int_{0}^{t} \sigma dW_s -\int_{0}^{t}\nabla g(X_s) ds +K_t, \\ Y_t &=& Y_0 + \int_{0}^{t} \sigma dW_s -\int_{0}^{t}\nabla g(Y_s) ds + \hat{K_t}
    \end{eqnarray*}
    with initial laws $\mu_0, \nu_0 \in \mathcal{P}_2(\bbR^d).$
    Let $X_t^n, Y_t^n$ be the penalized solutions in \eqref{Xn}, also driven by the same process $W_t$ and initial laws,
    \begin{eqnarray*}
        X_t^n &=& X_0 + \int_{0}^{t} \sigma dW_s -\int_{0}^{t}\nabla g(X_s^n) ds +K_t^n, \\ Y_t^n &= &Y_0 + \int_{0}^{t} \sigma dW_s -\int_{0}^{t}\nabla g(Y_s^n) ds + \hat{K_t^n}.
    \end{eqnarray*}
   Then, for all $t \geq 0$, we have the following contraction result 
    \begin{equation*}
        W_2(\mu_t, \nu_t) \leq e^{-mt}W_2(\mu_0, \nu_0) + K\left( \frac{\log n}{n} \right)^{1/4}
    \end{equation*}
    where $\mu_t$ and $\nu_t$ are the respective laws of $X_t$ and $Y_t,$ and $K > 0.$
\end{thm}

\begin{proof}
    Since $\mathcal{P}_2(\bbR^d)$ is a complete metric space under the Wasserstein-2 distance, the triangle inequality gives us \[
    W_2(\mu_t, \nu_t) \leq W_2(\mu_t, \mu_t^n) + W_2(\mu_t^n, \nu_t^n) + W_2(\nu_t^n, \nu_t),
    \]
    where $\mu_t^n$ and $\nu_t^n$ are the laws of $X_t^n$ and $Y_t^n$ respectively. Then, by Theorem \ref{penalized contraction} \[
    W_2(\mu_t^n, \nu_t^n) \leq e^{-mt}W_2(\mu_0, \nu_0).
    \]
    Using Theorem 4.1 of \cite{slominski2013}, under assumptions (A1) and (A2), we have that there exists constants $C_1$ and $C_2$ such that \begin{eqnarray*}
        \mathbb{E}\left[\sup_{0 \leq s \leq t} ||X_s^n -X_s||^2\right]^{\frac{1}{2}}  &\leq& C_1 \left(\frac{\log n}{n}\right)^{\frac{1}{4}}, \\ \mathbb{E}\left[\sup_{0 \leq s \leq t} ||Y_s^n -Y_s||^2\right]^{\frac{1}{2}}  &\leq& C_2 \left(\frac{\log n}{n}\right)^{\frac{1}{4}}.
    \end{eqnarray*}
    
    By \eqref{eq:W2_expectation}, for any coupling we know $(W_2(\mu_t^n, \mu_t))^2 \leq \mathbb{E}||X_t^n - X_t||^2.$ Since $||X_t^n - X_t|| \leq \sup_{s\leq t}||X_s^n - X_s||,$ we get \[
    W_2(\mu_t^n, \mu_t) \leq C_1 \left( \frac{\log n}{n} \right)^{1/4}, \quad W_2(\nu_t^n, \nu_t) \leq C_2 \left( \frac{\log n}{n} \right)^{1/4}.
    \]
    Putting the bounds together, \begin{equation*}
        W_2(\mu_t, \nu_t) \leq e^{-mt}W_2(\mu_0, \nu_0) + (C_1 + C_2)\left( \frac{\log n}{n} \right)^{1/4}.
    \end{equation*}
\end{proof}

\begin{cor}
    Under the assumptions of Theorem \ref{error contraction}, \begin{equation*}
        W_2(\mu_t, \nu_t) \leq e^{-mt}W_2(\mu_0, \nu_0), \quad t \geq 0.
    \end{equation*}
\end{cor}

\begin{proof}
    Let $n \to \infty$ in Theorem \ref{error contraction}. Since $(\log n/n)^{1/4} \to 0,$ the error term vanishes.
\end{proof}

For reflected Langevin SDE \eqref{X} \begin{equation*}
    X_t =X_0 + \int_{0}^{t} \sigma dW_s -\int_{0}^{t}\nabla g(X_s) ds +K_t,
\end{equation*}
there has been recent work to determine under which assumptions an RSDE admits a unique invariant probability measure $\pi$. The full details are contained in \cite{Eberle2019} and \cite{Lamperski2021}. In particular, for reflected stochastic differential equations satisfying particular growth conditions in III.B of \cite{Lamperski2021}, there exists constants $c > 0$ and $C \geq 1$ such that for all $t \geq 0$ and initial laws $\mu_0 \in \mathcal{P}_2(\bbR^d),$ there is a unique invariant measure $\pi$ such that \begin{equation}\label{inv Xt}
    W_2(\mu_t, \pi) \leq Ce^{-ct}W_2(\mu_0, \pi).
\end{equation}
The growth conditions in III.B of \cite{Lamperski2021} for \eqref{X} are fulfilled for $g \in C^1(\bbR^d)$ under assumptions (A1) and (A2) from Theorem \ref{error contraction}. 

\begin{rem}
A direct computation using the stationary Fokker--Planck equation with
the Neumann (zero‐flux) boundary condition shows that the reflected Langevin SDE \eqref{X} has invariant density proportional to $e^{-\frac{2}{\sigma^2}g(x)}\mathbf{1}_{\{x\in D\}}$. Thus, the invariant measure $\pi$ of the reflected Langevin SDE coincides with the intended constrained Gibbs
measure associated with the Langevin equation \eqref{diffusion}. Since existence and uniqueness of the reflected SDE is guaranteed under our assumptions on $D$, this identifies the law of the reflected process with the truncated Gibbs measure \cites{Tanaka1979, Lions1984}.
\end{rem}

\begin{thm}\label{final bound}
Under the assumptions of Theorem \ref{error contraction}, let $\pi$ denote the unique invariant measure of the reflected SDE \eqref{X}, and for $n \geq 1$, let $\pi^n$ denote the unique invariant measure of the penalized process \eqref{Xn}, as in Corollary \ref{inv Xn}. Then, there exists constants $C > 0$ and $\delta > 0$ such that
\[
W_2(\pi, \pi^n) \leq C n^{-\delta}.
\]
\end{thm}

\begin{proof}
By the triangle inequality, we have
\[
W_2(\pi, \pi^n) \leq W_2(\mu_t, \pi) + W_2(\mu_t^n, \pi^n) + W_2(\mu_t^n, \mu_t),
\]
where $\mu_t$ and $\mu_t^n$ are the laws of the reflected and penalized processes at time $t$, respectively. 

By \eqref{inv Xt} and Corollary \ref{inv Xn}), at time $t$,
\[
W_2(\mu_t, \pi) \leq C_1 e^{-ct} W_2(\mu_0, \pi), \quad W_2(\mu_t^n, \pi^n) \leq e^{-mt} W_2(\mu_0, \pi^n),
\]
where $c, C_1 > 0$ are independent of $n$. Note that $W_2(\mu_0, \pi^n)$ is uniformly bounded in $n$ since \[ f_n(x) = g(x) + \frac{n}{2}\text{dist}^2(x, D) \geq \frac{m}{2}\|x\|^2 - k\] for some constant $ k > 0 $ by the convexity of $f_n$ by Lemma \ref{f convex thm}. Then
\[
\int_{\mathbb{R}^d} \|x\|^2 \, d\pi^n(x) \leq K
\]
for some $K > 0 $ independent of $n$. Since $ \mu_0 \in \mathcal{P}_2(\mathbb{R}^d) $, it follows that $ W_2(\mu_0, \pi^n) $ is uniformly bounded in $n$. Thus, there exists $C_2 > 0$ such that $W_2(\mu_t^n, \pi^n) \leq C_2e^{-mt}.$ From Theorem 4.1 of \cite{slominski2013}, we have
\[
W_2(\mu_t^n, \mu_t) \leq C_3 \left( \frac{\log n}{n} \right)^{1/4},
\]
for some constant $C_3 > 0$. Combining all terms gives
\[
W_2(\pi, \pi^n) \leq C_1 e^{-ct} + C_2 e^{-mt} + C_3 \left( \frac{\log n}{n} \right)^{1/4}.
\]
Now, choose $t = \frac{1}{\kappa} \log n$ for $\kappa = \max(c, m)$. Since $\left( \frac{\log n}{n} \right)^{1/4} = o(n^{-\gamma})$ for any $\gamma < \frac{1}{4}$, let $\delta < \min\left( \frac{c}{\kappa}, \frac{m}{\kappa}, \frac{1}{4}\right)$. Then, right-hand side is bounded by $C n^{-\delta}$ for suitable $C > 0$.
\end{proof}

\section{Discrete Approximation of the Penalized Langevin Diffusion}
\subsection{Penalized Constrained Unadjusted Langevin Algorithm (PCULA)}

Having established the convergence of the continuous-time penalized process to the invariant probability measure of the reflected Langevin SDE, we now introduce a time-discretized version of the penalized SDE \eqref{Xn}. Our standing assumptions for the discretization follow from (A1) and (A2) of Theorem \ref{error contraction}: 
\begin{enumerate}
    \item[(EM1)] $g \in C^1(\bbR^d)$ is $m$-strongly convex and $\nabla g$ is $L$-Lipschitz.
    \item[(EM2)] $D \subset \bbR^d$ is nonempty, closed, and convex; $\psi(x) \coloneqq \frac{1}{2}\dist^2(x, D)$ so that $\nabla\psi(x) = x - \Pi(x)$ and $\nabla\psi(x)$ is 1-Lipschitz.
    \item[(EM3)] There exists $R > 0$ such that \[
        \sigma^2 + ||\nabla g(x)||^2 \leq R(1 + ||x||^2) \quad \text{ for all } x \in \bbR^d.\]
\end{enumerate} 

Consequently, $f_n \coloneqq g(x) + n\psi(x)$ is continuously differentiable, $m$-strongly convex and $\nabla f_n$ is $L_n$-Lipschitz where $L_n \coloneqq L + n.$ We simulate the penalized SDE instead of the reflected SDE and let the penalty term $n$ control the constraint. Applying the Euler-Maruyama scheme to \eqref{standard_penalized} gives \begin{equation}\label{EM}
    X_{k+1}^{(n,h)} = X_k^{(n,h)} - h\nabla f_n(X_k^{(n,h)}) + \sigma\sqrt{h}\xi_k
\end{equation}
for a fixed step-size $h > 0$, penalty $n \in \bbN$, and $\xi_k \sim \mathcal{N}(0, I_d)$.

From Section 2 of \cite{Moulines2019}, under assumptions (EM1)-(EM3), if $h \leq \frac{1}{m + L_n}$, then the discrete process $X_k^{(n,h)}$ admits a unique invariant measure $\pi^{(n,h)}$ such that $W_2(\pi^{(n,h)}, \pi^n) \leq Ch$ for some $C > 0$ where $\pi^n$ is the invariant measure of the continuous process $X_t^n$ from \eqref{standard_penalized}. Thus, if $\pi$ is the invariant measure of the reflected Langevin SDE \eqref{X}, then by Theorem \ref{final bound}, there exists constants $C_1, C_2, \delta > 0$ such that
\begin{equation}
    W_2\left(\pi, \pi^{(n,h)}\right) \leq C_1h + C_2n^{-\delta}
\end{equation} 

The bound on $W_2\left(\pi, \pi^{(n,h)}\right)$ ensures that the invariant measure of the discrete penalized process estimates that of the reflected Langevin SDE, subject to an appropriate step size $h$ and sufficiently large penalty parameter $n$. However, the constants $C_1, C_2$ and $\delta$ depend on the structural properties of the problem. More specifically, they depend on $m, L, d, \sigma,$ and on $n$ since $f_n$ is $L_n$-Lipschitz. Thus, increasing $n$ will tighten the constraint, but also increase the Lipschitz constant which in turn restricts the step size $h$. The proof of Theorem \ref{final bound} also shows that the bound depends on $t \gtrsim \log n$, so the number of iterations $k$ must grow at least logarithmically with $n$. Explicit formulations of the appropriate constants and parameters for Langevin methods is an active area of investigation; in particular, the step size $h$ need not stay fixed, and can be replaced by a non-increasing sequence to obtain stricter bounds. Recent related work is contained in \cites{Durmus2019, Eberle2019, Moulines2019}.

We now obtain the Penalized Constrained Unadjusted Langevin Algorithm from \eqref{EM}. This follows the framework of the Unadjusted Langevin Algorithm, with a modified potential energy function to enforce the boundary constraint. In particular, to sample from a target distribution $\pi$ supported on $D$, as in Section 1, we let $g(x) = \frac{-\sigma^2}{2}\log \pi(x).$

\begin{algorithm}[H]
\caption{Penalized Constrained Unadjusted Langevin Algorithm (PCULA)}
\begin{algorithmic}[1]
\State \textbf{Input:} Step size $h > 0$, penalty $n \in \bbN$, iterations $K$
\State \textbf{Initialize:} $X_0^{(n,h)} \in \bbR^d$
\For{$k = 0$ to $K-1$}
    \State Sample $\xi_k \sim \mathcal{N}(0, I_d)$
    \State Update:
    \[
    X_{k+1}^{(n,h)} = X_k^{(n,h)} - h\nabla f_n\left(X_k^{(n,h)}\right) + \sigma\sqrt{h}\xi_k
    \]
\EndFor
\State \textbf{Output:} $X_1^{(n,h)}, \cdots, X_K^{(n,h)}$
\end{algorithmic}
\end{algorithm}

\subsection{Truncated Gaussian Distribution via PCULA and Further Work}
To illustrate the behavior of the Penalized Constrained Unadjusted Langevin Algorithm (PCULA), we apply it to a two–dimensional example with a Gaussian distribution truncated to a closed convex domain as the target. Specifically, consider the domain
\[
D = \left\{ x=(x_1,x_2)\in\mathbb{R}^2 : \frac{x_1^2}{a^2} + \frac{x_2^2}{b^2} \le 1 \right\},
\]
an ellipse with semi–axes \(a = 1\) and \(b = \frac{1}{2}\), the diffusion coefficient $\sigma = 1,$ and the unnormalized Gaussian distribution
\[
\pi(x) \;\propto\; e^{-\alpha||x||^2/2}\,\mathbf{1}_{\{x\in D\}}
\] where the parameter $\alpha > 0$ controls the drift strength in the Langevin equation, with higher values yielding a more concentrated distribution.
\begin{center}
\begin{figure}[t]
    \centering
\includegraphics[width=0.65\linewidth]{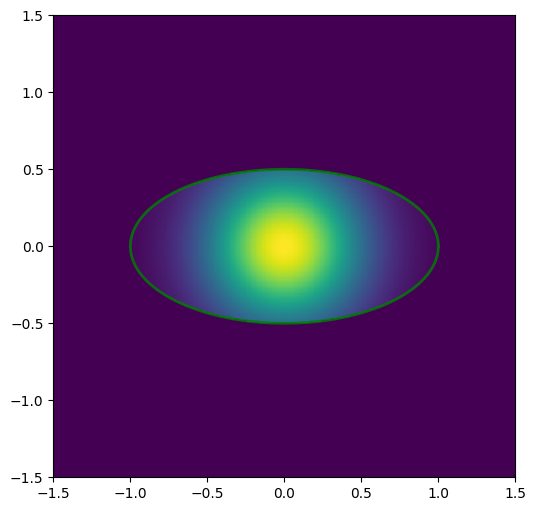}
    \caption{Truncated Gaussian supported on elliptical domain.}
    \label{fig:true-density}
\end{figure}
\end{center}
It must be noted that we do not need to compute the normalization factor as \eqref{diffusion} only requires $\nabla \log \pi(x)$.
This corresponds to choosing the potential 
\[
g(x)=\frac{\alpha}{4}\|x\|^2 ,
\]
and we construct the penalized energy function
\(f_n(x)=g(x)+n\psi(x)\) to use in our computations. Running PCULA for various values of \(n\) and fixed step size \(h\) 
demonstrates visually how the empirical distribution converges toward the truncated Gaussian, in Figure 2. The accuracy is defined as the percentage of samples within the domain $D.$
\begin{figure}[t]
    \centering
    
    \begin{subfigure}{0.48\linewidth}
        \centering
        \includegraphics[width=\linewidth]{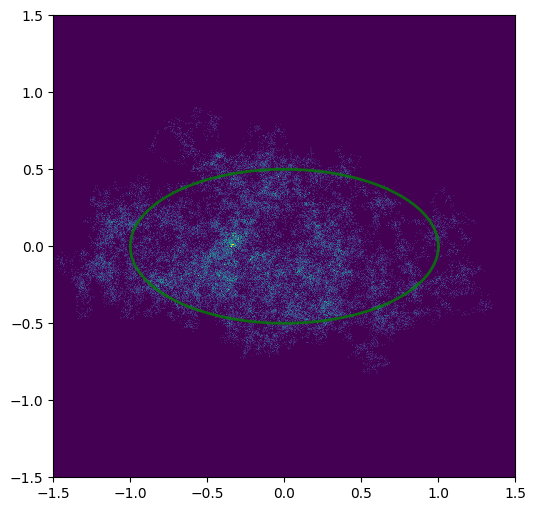}
        \caption{$n = 1,$ accuracy $= 71.9\%$.}
        \label{fig:pc1}
    \end{subfigure}
    \hfill
    \begin{subfigure}{0.48\linewidth}
        \centering
        \includegraphics[width=\linewidth]{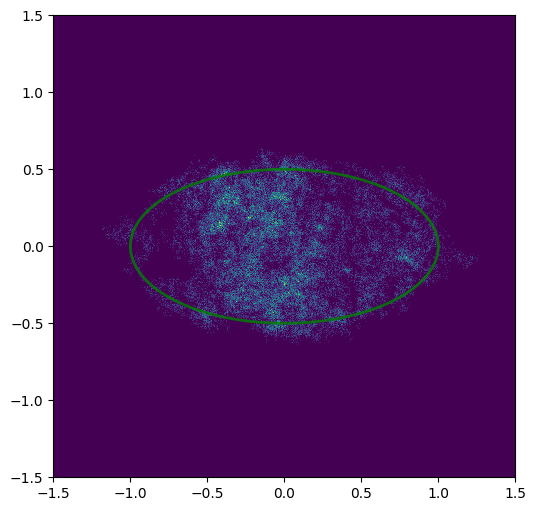}
        \caption{$n = 10,$ accuracy $= 85.8\%$.}
        \label{fig:pc10}
    \end{subfigure}

    \vspace{1em}

    \begin{subfigure}{0.48\linewidth}
        \centering
        \includegraphics[width=\linewidth]{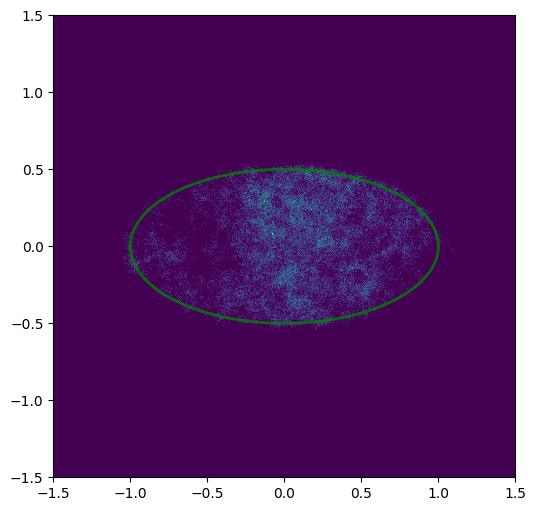}
        \caption{$n = 100,$ accuracy $= 95.3\%$.}
        \label{fig:pc100}
    \end{subfigure}
    \hfill
    \begin{subfigure}{0.48\linewidth}
        \centering
        \includegraphics[width=\linewidth]{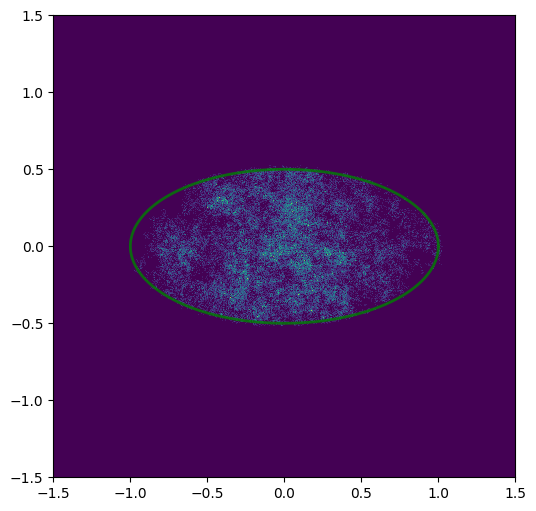}
        \caption{$n = 500,$ accuracy $= 98.5\%$.}
        \label{fig:pc500}
    \end{subfigure}

    \caption{Empirical density of truncated Gaussian distribution via PCULA with $10^5$ iterations and step size $h = 10^{-4}.$}
    \label{fig:pcula-grid}
\end{figure}

In this work, we demonstrated that boundary constraints in the Langevin
diffusion can be enforced through the reflected Langevin SDE \eqref{X}.  
To approximate the reflection mechanism, we introduced a family of penalized
SDEs \eqref{standard_penalized} incorporating a restoring force that pushes the
process back toward the domain $D$, and we proved that their invariant measures
converge to the unique invariant law of the reflected diffusion at a polynomial
rate in the penalization parameter. Discretizing the penalized dynamics via the Euler–Maruyama scheme yields the Penalized Constrained Unadjusted Langevin Algorithm (PCULA), which follows the standard ULA framework but replaces the original potential $g$ with the penalized energy $f_n$ for sufficiently large $n$.

There are several directions remaining for further exploration. An important question is whether penalization-based methods can be extended to
non-convex domains, as most existing constrained Langevin algorithms are
designed for convex bodies; recent work such as \cites{Sato2025} provides a
promising starting point for this case. Another direction is to explore alternative penalization functions or generalized penalty structures, and to optimize the choice of discretization parameters in the Euler–Maruyama scheme, following ideas developed in \cite{Durmus2019} and \cite{Gurbuz2024}. Further extensions to higher-order integrators or to underdamped Langevin dynamics also present new avenues of investigation.

\newpage

\section*{Acknowledgments}
The authors thank Professor Ludovic Tangpi from Princeton University for numerous insightful conversations that helped lead to the development of this work.

\end{document}